\documentclass{amsproc}

\newtheorem{theorem}{Theorem}[section]

\numberwithin{theorem}{section}
\newtheorem{proposition}[theorem]{Proposition}
\newtheorem{big-problem}[theorem]{Big Problem}

\newtheorem{conjecture}[theorem]{Conjecture}
\theoremstyle{definition}

\theoremstyle{definition}
\newtheorem{definition}[theorem]{Definition}
\newtheorem{example}[theorem]{Example}

\theoremstyle{remark}
\newtheorem{remark}[theorem]{Remark}

\numberwithin{equation}{section}

\usepackage{graphicx}
\usepackage[colorlinks=true, allcolors=blue]{hyperref}
\usepackage{mathtools} 
\allowdisplaybreaks 
\usepackage{comment} 
\usepackage[capitalize]{cleveref}
\usepackage{bbm}
\usepackage{nicematrix}

\newcommand{\R}{\mathbb{R}}
\newcommand{\N}{\mathbb{N}}
\newcommand{\Z}{\mathbb{Z}}
\newcommand{\C}{\mathbb{C}}
\newcommand{\PP}{\mathbb{P}}

\DeclareMathOperator{\rank}{rank} 

\begin{document}

\title{An Early History of Toric Ideals}

\author{Serkan Ho\c{s}ten}
\address{Department of Mathematics,
San Francisco State University, San Francisco, CA, USA}
\email{serkan@sfsu.edu}

\subjclass[2020]{14M25, 14Q15, 13P10, 13P25}
\date{Jan 20, 2025}

\keywords{toric ideals, Gr\"obner bases, initial ideals, integer programming, regular triangulation, Gr\"obner fan}

\begin{abstract}
Toric ideals are everywhere. They have been in the commutative algebra lexicon since about 1990 when Bernd Sturmfels used the term. The early days of toric ideals and their Gr\"obner bases were full of new results and promising developments in their applications. Bernd has been consistently their biggest promoter through his own work and that of his collaborators and students. This article is a personal and subjective recalling of the first decade of toric ideals when Bernd played a central role. 
\end{abstract}

\maketitle

\section{Introduction}
I arrived to the United States as a graduate student at Cornell University in August 1992. I came to study operations research, in particular, discrete optimization such as integer linear programming and network flow algorithms. As the typical first year in the PhD program goes, I took foundational courses in optimization (including a course taught by Mike Todd in interior point methods for linear programming that was then the rage) and other core courses in stochastic processes and statistics. At the time, each first year student in the OR Department had to thoroughly read and then present to a committee of three faculty members a published research article of their choice. I chose Hendrik Lenstra's paper on solving integer programs in polynomial time in fixed dimension \cite{Lenstra83}. I was certain that I would work on algorithms and computation for integer programs (cutting planes, branch-and-bound, and all that) with Les Trotter. 

At Cornell, besides the major subject, one had to also take courses in two minor subjects. I knew one of my minor subjects would be computer science. What should have the other one been? In my first semester, I took a course in advanced linear algebra from the Math Department that was taught by Keith Dennis. I almost flunked it but I really liked Keith's teaching style, so when I found out he would be teaching an honors course in abstract algebra (groups, rings, etc.) I decided to take this course. Since my undergrad education was in engineering, I had no idea what the course would be about but I signed up anyway. And I {\it loved} it (and did very well). The natural choice for the second minor topic was emerging to be mathematics. What other courses would have been good to take then? I knew I would do the year-long graduate algebra sequence (to be taught by Mark Gross), but what else was out there?

Rekha Thomas was also a graduate student in the OR Department who was couple years ahead of me. My memory is dark about this point, but maybe I heard Rekha giving a talk on her work about Gr\"obner bases of toric ideals and solving integer programs or I simply stroke up a conversation with her about the same topic. We had a long and fascinating discussion. She told me in detail what she was doing and it looked beautiful to me. It looked like one can do both optimization and algebra at the same time. A perfect fit! She also told me that in Fall 1993 her advisor, a certain Bernd Sturmfels, would be teaching a course on Gr\"obner bases. It was clear that I had to take this course and also learn about toric ideals.

\section{Toric Ideals and their Gr\"obner Bases}\label{sec:toric-ideals}

Bernd used the first edition of {\it Ideals, Varieties, and Algorithms} \cite{CLO} that was published just the previous year. Agnes Szanto was also in the course and  Hal Schenck was the grader. In fact, Bernd was also teaching a course on toric varieties during the same semester. I knew zero algebraic geometry so I stuck with Gr\"obner bases.

Bernd had been thinking and publishing about Gr\"obner bases already since 1989.
His paper co-authored with Neil White on invariant theory and Gr\"obner bases \cite{SW} is seminal. Equally important is his solo paper from 1990 where he proved the following famous result. See also the chapter by Aldo Conca in this volume.  
\begin{theorem}\cite[Theorem 1]{S90} 
The set of $r$-minors of a generic $m \times n$ matrix $ x_{ij})$ is a reduced Gr\"obner basis of the ideal they generate with respect to the lexicographic order induced by 
$$ x_{1,n} \succ x_{1,n-1} \succ \cdots \succ x_{1,1} \succ x_{2, n} \succ x_{2, n-1} \succ \cdots \succ x_{2,1} \succ x_{m,n} \succ x_{m, n-1} \succ \cdots \succ x_{m,1}.$$
\end{theorem}
This brings us to the paper that started all regarding toric ideals and their Gr\"obner bases, namely {\it Gr\"obner bases of toric varieties} \cite{S91}. 
Although it is published in 1991, it was submitted in March 1990. This means that Bernd must have known the contents of the article sometime in late 1989 or early 1990. The only reason this fundamental article has only 38 citations on the MathSciNet is because everything in it and a lot more appears in Bernd's book {\it Gr\"obner bases and convex polytopes} \cite{S96} (949 citations in January 2025). 
\begin{definition}
Let $A = [a_1 \, a_2 \, \cdots \, a_n]$ be a $d \times n$ matrix with integer entries and $\rank(A)=d$. Consider the $\C$-algebra homomorphism
$$\phi_A \, : \, \C[x_1, \ldots, x_n] \longrightarrow \C[t_1^{\pm}, t_2^{\pm}, \ldots, t_d^{\pm}] \quad \quad \phi_A(x_i) = t_1^{a_{1i}}t_2^{a_{2i}}\cdots t_d^{a_{di}}.$$
The toric ideal $I_A$ is the kernel of $\phi_A$.
\end{definition}

\begin{remark} \label{rmk:A-matrix}
The above is the most general definition of a toric ideal. Typically, we want $\mathrm{cone}(A)$, the rational polyhedral cone generated by $a_1, \ldots, a_n$, to be pointed. This guarantees that the toric ideal $I_A$ does not contain binomials of the form $x^u - 1$. Furthermore, if we assume that
the monoid $\mathrm{cone}(A) \cap \N^d$ is
generated by $a_1, \ldots, a_n$, i.e. it is equal to $\N(A) := \{u_1a_1 + \cdots + u_na_n \, : \, u_1, \ldots, u_n \in \N \}$, then 
$I_A$ is the defining ideal of the affine (normal) toric variety $X_A$ . Finally, if we stipulate that the first row of $A$ is $[1 \, 1 \, \cdots \, 1]$, or more generally, the row span of $A$ contains this vector, then $I_A$ is a homogeneous ideal defining the projective toric variety $X_A$ in $\PP^{n-1}$; see \cite{S96, CLS}. In the rest of this chapter I will stay in this last setting as Bernd did in \cite{S91}.
\end{remark}
\begin{theorem} \cite[Lemma 2.2 and Lemma 2.5]{S91}
$I_A$ is a prime ideal generated by the binomials 
$$ \{ x^u - x^v \, : \, u-v \in \ker(A) \cap \Z^n \}.$$
Any reduced Gr\"obner basis of $I_A$ also consists of binomials. 
\end{theorem}
\begin{example} \label{ex:3x3x3}
Let 
$$A = [e_i \oplus e_j \oplus e_k \, : \, 1\leq i,j,k \leq 3], $$
where $e_1, e_2, e_3$ form the standard unit basis of $\R^3$. The corresponding 
monomial map is
$$\phi_A \, : \, \C[x_{111}, \ldots, x_{333}] \longrightarrow \C[a_1,a_2,a_3, b_1, b_2, b_3, c_1, c_2, c_3] \quad \quad \phi_A(x_{ijk}) = a_ib_jc_k, \, \, i,j,k=1,2,3.$$
The toric ideal $I_A$ is the defining ideal of the Segre embedding of $\PP^2 \times \PP^2 \times \PP^2$ in $\PP^{26}$. 
It is minimally generated by $162$ quadratic binomials. They form the reduced Gr\"obner of $I_A$ with respect to the degree reverse lexicographic order induced by 
$$ x_{111} \succ x_{112} \succ x_{113} \succ \cdots \succ x_{133} \succ x_{211} \succ x_{212} \succ x_{213} \succ \cdots \succ x_{233} \succ x_{311} \succ x_{312} \succ x_{313} \succ \cdots \succ x_{333}.$$
They are the union of the  $2$-minors of the following two matrices:
$$\begin{bmatrix}
 x_{111} &  x_{112} &  x_{113} &  x_{121} &  x_{122} & x_{123} &  x_{131} & x_{132} &  x_{133} \\
 x_{211} &  x_{212} &  x_{213} &  x_{221} &  x_{222} & x_{223} &  x_{231} & x_{232} &  x_{233} \\
 x_{311} &  x_{312} &  x_{313} &  x_{321} &  x_{322} & x_{323} &  x_{331} & x_{332} &  x_{133} \end{bmatrix},$$
$$\begin{bmatrix}
x_{111} & x_{121} & x_{131} &  x_{211} & x_{221} & x_{231} &  x_{311} &  x_{321} &  x_{331} \\
x_{112} & x_{122} & x_{132} &  x_{212} & x_{222} & x_{232} &  x_{312} &  x_{322} &  x_{332} \\
x_{113} & x_{123} & x_{133} &  x_{213} & x_{223} & x_{233} &  x_{313} &  x_{323} &  x_{333} 
\end{bmatrix}.
$$
We note that these matrices are two particular flattenings of the generic $3 \times 3 \times 3$ tensor $(x_{ijk})$.
\end{example}

Next, Bernd proved a bound on the degree of any reduced Gr\"obner basis element of $I_A$. This bound is singly-exponential in the input data, namely, in $d$, $n$, and the size of entries of $A$. It contrasts with a previous famous result 
which had shown that a binomial (although not toric) ideal could have reduced Gr\"obner basis elements with degrees doubly-exponential in the degrees of the input binomials \cite{MM82}. To state and prove the result, I introduce a few things which will be useful later on as well. For each sign pattern $\rho \in \{+1,-1\}^n$ we let $\R_\rho^n := \{x \in \R^n\, : \, \rho_i x_i \geq 0, \, i=1, \ldots, n\}$ and $H_\rho := \R_\rho^n \cap \ker_\Z(A)$. The latter is a pointed rational polyhedral cone, and it is generated by circuits of $A$ that are in $H_\rho$. A circuit of $A$ is a nonzero vector $u \in \ker_\Z(A)$ with relatively prime entries and with minimal support. The cardinality of the support of any circuit is at most $d+1$. If $\{i_1, \ldots, i_{d+1} \} \subset [n]$ contains the support of a circuit $u$ then
\begin{equation} \label{eq:circuit}
 u_{i_j} =  \pm \frac{1}{c} \det(a_{i_1}, \ldots, a_{i_{j-1}}, a_{i_{j+1}}, \ldots, a_{i_{d+1}})
 \end{equation}
where $c$ is the gcd of all $(d+1)$
determinants appearing in the above formula. The following degree bound theorem is a slight improvement of the one that is in \cite{S91} and appears in \cite{S96}.
\begin{theorem} \cite[Theorem 2.3]{S91} \cite[Theorem 4.7]{S96} \label{thm:degree}
Let $D(A)$ be the maximum of the absolute value of all $d$-minors of $A$. The degree of any reduced Gr\"obner basis element of $I_A$ is at most $(n-d)(d+1)D(A)$.   
\end{theorem}
\begin{proof}
 Let $x^{v_+} - x^{v_-}$ be a reduced Gr\"obner basis element of $I_A$ where $v = v_+ - v_- \in \ker_\Z(A)$. Then $v$
 is in one of $H_\rho$. Since the dimension of $H_\rho$ is $n-d$, Carath\'{e}odory's theorem implies that 
 $v$ is a conical combination of at most 
 $n-d$ circuits in $H_\rho$:
 $$ v = \lambda_1 u_1 + \cdots + \lambda_{n-d} u_{n-d}.$$
 We claim each $\lambda_i \leq 1$. Otherwise, let $w =\sum_{i=1}^{n-d} (\lambda_i - \lfloor \lambda_i \rfloor) u_i$. We see that $ v = w + (v-w)$ and observe that both $w$ and $v-w$  are in $H_\rho$. But then if we write $w = w_+ - w_-$, then $x^{w_+} - x^{w_-}$ is in 
 $I_A$, and $x^{w_+}$ divides $x^{v_+}$ and $x^{w_-}$ divides $x^{v_-}$. This is a contradiction to our assumption that $x^{v_+} - x^{v_-}$ is in a reduced Gr\"obner basis. Since any circuit has 
 degree at most $(d+1)D(A)$ we conclude 
 that the degree of our binomial is at
 most $(n-d)(d+1)D(A)$.
\end{proof}
There are two other important results in \cite{S91} which spurred subsequent work. 
One is the construction of the {\it state polytope} of $I_A$. The other one is the characterization of the radical of any initial ideal of $I_A$. I briefly present these results.

Let $\{x^{u_1} - x^{v_1}, \,  \ldots, \,  x^{u_k} - x^{v_k}\}$ be a reduced Gr\"obner basis of $I_A$ where we assume that the first term in each binomial is the initial term. We consider the {\it Gr\"obner cone}
$\mathcal{C} = \{\omega \in \R^n \, : \, 
\omega \cdot u_i \geq \omega \cdot v_i, \, i=1, \ldots, k\}$. It turns out that although there are infinitely many term orders in a polynomial ring with at least two variables, every ideal has finitely many distinct reduced Gr\"obner bases \cite[Theorem 1.2]{S96}. Moreover, the collection of all Gr\"obner cones $\mathcal{C}$, one for each reduced Gr\"obner basis of $I_A$, is a complete polyhedral fan, called the {\it Gr\"obner fan} $\mathcal{G}_A$. The existence of Gr\"obner fans for any 
homogeneous polynomial ideal had been proved \cite{BM88, MR88}. Moreover, these fans are polytopal. In other words, there exists a polytope $\mathcal{S}_A$, called the {\it state polytope} of $I_A$ whose normal fan equals the Gr\"obner fan $\mathcal{G}_A$ \cite{BM88}.  
\begin{theorem} \cite[Proposition 2.7]{S91} Let $E_A$ be the zonotope generated by all circuits of $A$. The state polytope  $\mathcal{S}_A$ of the toric ideal $I_A$ is a Minkowski summand of 
the zonotope 
$\sum \{[u_+, u_-] \, : \, u \in E_A \cap \Z^n\}$.
\end{theorem}

For the second result, recall that a monomial ideal $I$ is radical if and only if its minimal generators are squarefree. Moreover, there is a bijection between the radical monomial ideals in $\C[x_1, \ldots, x_n]$ and simplicial complexes on $n$ vertices: given such a simplicial complex $\Delta$ we can construct the {\it Stanley-Reisner ideal}
$$I_\Delta := \bigcap_{\sigma \in \Delta} \langle x_i \, : i \not \in \sigma \rangle$$
where the intersection runs over all facets of $\Delta$. This can be reversed by computing the irredundant prime decomposition of the radical monomial ideal $I$. Given the matrix $A$, every sufficiently generic $\omega \in \R^n$ gives rise to a triangulation $\Delta_\omega$ of $\mathrm{cone}(A)$ into simplicial cones. We think of $\Delta_\omega$
as a simplicial complex on $n$ vertices where we declare a subset $\sigma \subset [n]$  a 
face of $\Delta_\omega$ if there is a vector $y = (y_1, \ldots, y_d) \in \R^d$
such that $a_i \cdot y = \omega_i$ if $i \in \sigma$ and $a_i \cdot y < \omega_i$ if $i \not \in \sigma$. The triangulation $\Delta_\omega$ is called the {\it regular triangulation} of $A$ induced by $\omega$; see \cite{S96, DSR10}. 
\begin{theorem} \cite[Theorem 3.1]{S91}
Let $\omega$ be a vector in the interior of a Gr\"obner cone $\mathcal{C}$ of the toric ideal $I_A$. The radical of the corresponding initial ideal is the Stanley-Reisner ideal $I_{\Delta_\omega}$ of the regular triangulation $\Delta_\omega$.
\end{theorem}

It is hard not to call this paper phenomenal. In the meantime, I was continuing to learn from Cox-Little-O'Shea \cite{CLO} as well as from Bernd and Rekha. 
\section{Integer programming} 

I continue using the $d \times n$ integer matrix $A$ under the assumptions of Remark \ref{rmk:A-matrix}. Now consider the 
discrete optimization problem $IP_{A,\omega}(b)$
$$ \mathrm{minimize} \,\,  \omega \cdot x  
\mbox{  subject to  } Ax = b \mbox{ and  } x \in \N^n$$
where $b \in \N(A)$. Solving such a general integer program is NP-hard, and already by early 90s there had been substantial amount of work accumulated on all sorts of methods (but mostly polyhedral) to solve these hard problems; see \cite{Sch86} from those times which is still a crucial reference source. In what follows, the problem gets even harder since we will attempt to solve the above problem for all possible right-hand-side vectors $b$ while fixing $\omega$. I will denote this family of integer programs as $b$
varies by $IP_{A,\omega}$.

A method to accomplish this using Gr\"obner bases of the toric ideal $I_A$ had been outlined in \cite{CT91}:
\begin{enumerate}
    \item Form the ideal 
    $$J_A = \langle x_1 - t^{a_1}, \, x_2 - t^{a_2}, \ldots, x_n - t^{a_n} \rangle \subset \C[t_1, \ldots, t_d, x_1, \ldots, x_n].$$
    \item Compute the reduced Gr\"obner basis
    $\tilde{\mathcal{G}}_{A, \succ}$ of $J_A$ with respect to the elimination term order where $t_j \succ x_i$ for every pair of $j = 1, \ldots, d$ and $i=1, \ldots, n$, and where $x^u \succ x^v$ if $\omega \cdot u > \omega \cdot v$. In the last case if equality occurs use a lexicographic term order to break the tie $x^u \succ x^v$.
    \item Compute the normal form $x^v$ of the monomial $t^b$ with respect $\tilde{\mathcal{G}}_{A, \succ}$. The vector $v$ is the optimal solution. 
\end{enumerate}
Note that, technically, we are not using a reduced Gr\"obner basis of $I_A$ but the reduced Gr\"obner basis $\tilde{\mathcal{G}}_{A,\succ}$ of the (toric!) ideal $J_A$. However, the results of the elimination theory imply that $\mathcal{G}_{A, \succ} = \tilde{\mathcal{G}}_{A,\succ} \cap \C[x_1, \ldots, x_n]$ is a reduced Gr\"obner basis of $I_A$. Another important note here is that once $\mathcal{G}_{A,\succ}$ is computed, one can determine $u \in \N^n$
such that $Au=b$ (still an NP-hard problem) and compute the normal form of $x^u$ with respect to $\mathcal{G}_{A,\succ}$. 

Here came Rekha's first fundamental contribution to the subject. The paper that would become \cite{T95} was completed in early 1994 and I knew its contents in the fall of 1993. She made two observations which tied the Gr\"obner basis theory of toric ideals to integer programming. First, she saw that $\mathcal{G}_{A,\succ}$
forms the unique minimal {\it test set} for our integer programs as $b$ varies \cite[Corollary 2.1.10]{T95}. A test set is any collection $\mathcal{T} \subset \ker_Z(A)$ such that 
\begin{itemize}
 \item[a)] for any non-optimal solution $v$ for any integer program $IP_{A,\omega}(b)$ in the family $IP_{A,\omega}$ there exists $w \in \mathcal{T}$ such that $x^v \succ x^{v-w}$, and
 \item[b)] for the optimal solution $u$ of $IP_{A,\omega}(b)$
 and for all $w \in \mathcal{T}$, $u-w$ is infeasible for $IP_{A,\omega}(b)$.
\end{itemize}
Second, she observed that on the set of feasible solutions of each $IP_{A,\omega}(b)$ one can create a graph $G_{A,\omega, b}$: the vertices of this graph are the feasible solutions of $IP_{A,\omega}(b)$, and there is a directed edge from $v$ to $u$ if 
$x^{(v-u)_+} - x^{(v-u)_-} \in \mathcal{G}_{A,\succ}$.
\begin{theorem} \cite[Theorem 2.1.8]{T95} Each $G_{A,\omega,b}$ is a connected acyclic directed graph whose unique sink is the optimal solution to $IP_{A,\omega}(b)$.
\end{theorem}
Finally, this paper contains a re-interpretation of the usual Buchberger's algorithm \cite{B70} in the case of toric ideals, stripping away the algebra of ideals and viewing a binomial $x^u- x^v$ as just $u-v$, an element of $\Z^n$. This was called the {\it geometric Buchberger's algorithm}. 

By the time I graduated in 1997, Rekha published multiple papers in quick succession extending the Gr\"obner basis technology in integer programming; see \cite{TTN95, TW96, ST97, TW97, HT98}. The one that stands out among these is {\it Variations of cost functions in integer programming} \cite{ST97}, co-authored with Bernd. The contents of this article were also in the air during late 1993, and the manuscript was completed in 1994. Among other things, this paper introduced the {\it Graver basis} of $I_A$ (or $A$). 
\begin{definition}
Let $\mathcal{H}_\rho$ be the Hilbert basis of the cone $H_\rho$.
The set of binomials 
$$Gr_A = \{x^u-x^v \, : \, u-v \in \bigcup_{\rho} \mathcal{H}_\rho\}$$ 
is called the Graver basis of $I_A$.
\end{definition}
Recall that the union of all reduced Gr\"obner bases of an ideal $I$ is called the {\it universal Gr\"obner basis} (since it is a Gr\"obner basis of $I$ for any term order). We will denote the universal Gr\"obner basis of $I_A$ with $\mathcal{U}_A$.
\begin{theorem} \cite[Theorem 2.7]{ST97} The Graver basis $Gr_A$ contains the universal Gr\"obner basis $\mathcal{U}_A$.   
\end{theorem}
\begin{theorem} \cite[Algorithm 4.5]{ST97}  \label{thm:graver} Let 
$$ \Lambda(A) = \begin{bmatrix}
    A & 0 \\ I_n & I_n
\end{bmatrix}$$
be the Lawrence lifting of $A$. Let $\mathcal{G}_{\Lambda(A)}$ be the reduced Gr\"obner basis of $I_{\Lambda(A)}$ with respect to any term order. Then $x^u - x^v$ is in $Gr_A$ if and only if $x^uy^v - x^vy^u$ is in $\mathcal{G}_{\Lambda(A)}$.  
\end{theorem}
Along the same lines, this article also proved that every binomial corresponding to a circuit of $A$ is in $\mathcal{U}_A$ \cite[Lemma 2.6]{ST97}, and if $A$ is a unimodular matrix then $\mathcal{U}_A$ consists of only the circuits \cite[Theorem 2.10]{ST97}. Finally, there was a close look at the Gr\"obner fan $\mathcal{G}_A$ and the state polytope $\mathcal{S}_A$. In particular, there was an intriguing and very concrete conjecture.
\begin{conjecture} \label{conj:3/4} \cite[Conjecture 6.2]{ST97} Let $A$ be an $d \times n$ matrix where $n-d = 3$. Then every maximal cone in the Gr\"obner fan $\mathcal{G}_A$ has three or  four facets. 
\end{conjecture}
Bernd and Rekha called this the ``3/4 conjecture". I liked this conjecture but I had no idea how to systematically approach it. Below, I will come back to its resolution.  Rekha was determined to make progress in computing Gr\"obner fans of toric ideals. She came back to it in the first half of the naughts. Using a reverse search enumeration technique, she implemented a {\tt MAPLE} code called {\tt TiGERS} (Toric Gr\"obner bases Enumeration using Reverse Search) to compute Gr\"obner fans. Then Birkett Huber, another student of Bernd from Cornell, joined her in this project and {\tt TiGERS} became even faster in its new implementation in {\tt C} \cite{HT00}. It could compute Gr\"obner fans with more than $200,000$ maximal cones. Two other papers followed \cite{FJLT07, FJT07} with Anders Jensen, Rekha's PhD student  among the co-authors. Anders produced an even better implementation of reverse enumeration and Gr\"obner walk in his code {\tt Gfan} \cite{gfan} which is still the current state of the art. {\tt Gfan} is available as a stand-alone program as well as a Macaulay 2 \cite{M2} package.

Eventually, I was involved in the resolution of Conjecture \ref{conj:3/4}. Bernd had moved to UC Berkeley in 1995 and I was a postdoc at MSRI (now SLMath) during 1997/98. During that academic year and beyond, I collaborated in a project with Diane Maclagan, one of Bernd's first students from Berkeley,  where we looked at a particular monomial ideal that we called the {\it vertex ideal} of $A$. While studying associated primes of vertex ideals, we constructed an interesting example associated to the corank 3 matrix $A = [15 \, 247 \, 248 \, 345]$. Gr\"obner fans were not on our mind at all. However, Diane happened to decide to compute the Gr\"obner fan of $I_A$ using {\tt TiGERS}; just for fun, I guess. 
\begin{theorem} \cite[Theorem 3.4]{HM02} Conjecture \ref{conj:3/4} is false. The Gr\"obner fan $Gr_A$
of $A= [15 \, 247 \, 248 \, 345]$
has a maximal cone associated to 
$\omega = (111, 0, 341, 1)$ that has five facets.
\end{theorem}
We even generated another example with six facets \cite[Example 3.3]{HM02}. A little later Bernd joined us to write a paper on {\it supernormal vector configurations} \cite{HMS04}. As a consequence of this work, examples of Gr\"obner cones of dimension three were found which had as many as $15$ facets. 

\section{Computing Gr\"obner bases of toric ideals} 

As 1993 came to a close two new developments happened. First, Bernd and Rekha, jointly with Jes\'us De Loera (Jes\'us was also Bernd's student, junior to Rekha but senior to me), combinatorially described a reduced Gr\"obner basis of the {\it second hypersimplex} $A_d$ for any $d$ 
\cite[Theorem 2.1]{DST95}. Here the matrix $A_d$ has $d$ rows and $n = \binom{d}{2}$ columns where 
the columns are formed by all $0/1$
vectors in $\R^d$ with exactly two ones. This was a first example of a family of toric ideals associated
to non-unimodular matrices for which one could describe a reduced Gr\"obner basis. 

More importantly (in hindsight!), Bernd was working on his seminal paper with Persi Diaconis where they  introduced the {\it Markov basis} (aka minimal generating set) of a toric ideal $I_A$ to run Monte Carlo Markov chains for sampling purposes \cite{DS98}. This paper was not published until 1998 and it arguably marks the true beginning of the field of {\it algebraic statistics}. I am not going to dwell on this topic since the chapter in this volume by Seth Sullivant does a superb job in telling that story. In any case, the explosion of the work based on \cite{DS98} had to wait for almost another decade. 

I was motivated by the optimism that was shared by Bernd, Jes\'us, and Rekha, that Gr\"obner bases of toric ideals should be part of the toolbox for solving integer programs. Of course, one had to first compute such Gr\"obner bases, and besides the elimination-based approach there was no other method. {\tt MAPLE} had an implementation of the Buchberger's algorithm for general polynomial ideals. {\tt Macaulay}, the precursor to {\tt Macaulay 2} \cite{M2}, was already around with its earlier versions (version 3.0 for {\tt Linux} came out in September 1994). It had a fast implementation of Buchberger's, but again for general polynomial ideals. There was a  need for a fast implementation of the geometric Buchberger's algorithm for toric ideals. 

Before the end of the semester, Bernd invited me for a meeting to his office in White Hall. He patiently went over the advantages for a stand-alone implementation. For instance, computing the normal form of a monomial $x^v$ with respect to a partial Gr\"obner basis, a crucial ingredient in the Buchberger's algorithm, could be made much more efficient: if such a basis had a binomial $x^w - x^u$ where the leading term is $x^w$ and it divided $x^v$, then a one-step reduction would produce $x^{v-w+u}$. But maybe the same binomial could be applied multiple times for a longer chain of reductions. So instead, one could determine the largest $k$ such that $v - k(w-u)$ did not have a negative entry and reduce $x^v$ 
to $x^{v-k(w-u)}$ in one shot. Bernd demonstrated these ideas on his SUN workstation (named ``Grassmann") running {\tt UNIX}. We brainstormed similar ideas, and he convinced me that it was worth the effort to produce an implementation. I had experience in coding algorithms (mostly network flow and matching algorithms) in {\tt C} in my undergrad days and I thought I could do it.   

I convinced Lisa Fleischer, another second year student from the OR Department taking Bernd's course on Gr\"obner bases, to join me in the effort. We spent the entire winter break coding and refining and improving. Our first goal was to implement the geometric Buchberger's algorithm for a toric ideal $I_A$ starting from a given generating set of $I_A$. Already by the end of the break Lisa was losing interest but I kept going. When I felt ready, Bernd (and Hyungsook) invited me for dinner over at their house. Rekha was there, too. Bernd asked to try the code. I dialed in into the workstation that I was working with from Bernd's home computer. It is perfectly possible that I am misremembering this, but I think he wanted to compute the Graver basis $Gr_A$ of the ideal in Example \ref{ex:3x3x3} using Theorem \ref{thm:graver}. We set up the input and pushed the button. It took couple of minutes,
way slow in today's standards, but
it produced the correct output, something Bernd could not get {\tt MAPLE} or {\tt Macaulay} to do. He said something to the effect that he was impressed.  

Throughout the spring 1994 I kept refining and optimizing the code with all the bells and whistles. Mike Stillman who was the creator of {\tt Macaulay} with Dave Bayer was also interested. He gave me more ideas. When David Eisenbud was visiting Cornell to give a talk I showed my code to him. He seemed to also like it. In the meantime, Bernd told Lisa and me that he could only take one more new student. To my relief, Lisa took herself out of the picture, and I became an official student of Bernd. So I started joining the group meetings consisting of Bernd, Jes\'us, and Rekha, as well as Paco Santos (who went on later to prove that toric Hilbert schemes could be disconnected and the Hirsch conjecture is false) and Anna Bigatti (of {\tt CoCoA} fame \cite{CoCoA}). Very exciting times indeed! The whole semester culminated in my first international conference and my first talk at the ``Hungarian-American Workshop on Combinatorics" in Budapest in May 1994. In the audience were people like Jon Lee (combinatorial optimization) and Satish Rao (computer science). The talk was received well and people who are not so familiar with this kind of work seemed to be intrigued. 

During the summer I kept working on the code, but it became clear that it reached its limit as it was. Some new idea was needed. In particular, it would have been much more desirable if one started without a generating set of $I_A$ and not do elimination, and yet compute a Gr\"obner basis. Bernd pointed out the following. 
\begin{proposition} \cite[Proposition 3 and Proposition 4]{HS95}  \label{prop:GRIN} Let $\{u^{(1)}, \, u^{(2)}, \ldots, \, u^{(n-d)} \}$ be a lattice basis for $\ker_\Z(A)$ and let $J = \langle 
x^{u^{(1)}_+} - x^{u^{(1)}_-}, \, \ldots, \,  x^{u^{(n-d)}_+} - x^{u^{(n-d)}_-} \rangle$. Then 
$$I_A \, = \, J \, : \, (x_1\cdots x_n)^\infty \, \, = ((\cdots (J \, : \, x_1^\infty) \, : \, x_2^\infty) \cdots) \, : \, x_n^\infty). $$
Moreover, each saturation on the right hand side of the formula can be computed by a single Gr\"obner computation with respect to a reverse lexicographic term order.  
\end{proposition}

This proposition gives the desired algorithm and the only input needed is the matrix $A$. Computing a lattice basis of $\ker_\Z(A)$ is easy via the Hermite normal form algorithm \cite[Section 5.3]{Sch86}. One can even compute a {\it reduced} lattice basis in the sense of the LLL-algorithm \cite{L86}. Then $n$ applications of the geometric Buchberger's algorithm are needed, starting from the ideal $J$ and going through all the intermediate binomial ideals, ending in $I_A$. Surprisingly, this worked and it worked really fast. 

During 1994/95, Bernd was on sabbatical at NYU in New York City. I visited him there and he came back to Ithaca to check on his students. Encouraged by the success of the latest implementation we decided to make the code official. We called it {\tt GRIN} (GR\"obner bases for INteger programming). The paper we wrote (my first !) \cite{HS95} got accepted to the "Integer Programming and Combinatorial Optimization" (IPCO) conference in May 1995. We tried to make the paper really appealing to the optimization community. Our computational experiments showed that for as large (!) as $8 \times 16$ matrices $A$ we could solve integer programs $IP_{A,\omega}(b)$ in comparable times with the fastest integer programming solver at the time ({\tt CPLEX}). 
Of course, we knew this would not scale up to solve integer programs with thousands or millions of variables. But this was a proof of concept. I have to say I was a bit disappointed with the reaction to my talk at IPCO in Copenhagen. It was mixed. Some, such as Bob Bixby, the author of {\tt CPLEX}, saw it as a fruitful direction to pursue. Others were less impressed, mostly because they saw that the algorithm would not scale. 

This disappointment affected me and the progress of {\tt GRIN}. The plan had been to set up easy input and output formats, to work out the kinks in the algorithm, and to keep improving. I did help others do computations with {\tt GRIN} but I never brought it to a place which was available to everyone. Maintaining a software for users looked daunting. Luckily, others picked up from there. For instance, 
you can compute the toric ideal $I_A$ in {\tt Singular} \cite{Singular} just from $A$. If you issue the command {\tt toric\_ideal(A,"hs")} (the option stands for Hosten-Sturmfels) the algorithm based on Proposition \ref{prop:GRIN} is used. 

More importantly, {\tt 4ti2}  \cite{4ti2} came to the scene in 2001. Raymond Hemmecke, while doing his PhD in Duisburg, developed code for computing Graver bases solely based on computing Hilbert bases of rational cones. With the help of others, such as Ralf Hemmecke, Matthias K\"oppe, Peter Malkin, and Matthias Walter, {\tt 4ti2} became the go-to software for computing toric ideals, their Markov bases, Gr\"obner bases, Graver bases, and much more. The code helped significantly to produce results in stochastic integer programming \cite{HemS03, AH07, DHOW08}, lattice point counting \cite{DHTY04}, and probability theory and algebraic statistics \cite{HMSS08, BHLS10, HLS12}. {\tt Macaulay 2} \cite{M2}
has an interface with {\tt 4ti2} as well. Raymond was a postdoc at UC Davis in early naughts with Jes\'us as {\tt 4ti2} was developed and perfected. I believe Jes\'us deserves a lot of credit for carrying the torch, keeping in touch with the discrete optimization community, and advocating for algebraic methods in discrete optimization. His chapter in this volume is a wonderful resource as well as the book he co-authored with Raymond and Matthias \cite{DHK13}.

\section{Gr\"obner bases and convex polytopes} 

We arrived at the peak event of all of these developments, namely the ``Holiday Symposium" at New Mexico State University (NMSU), Las Cruces, that was held December 27-31, 1994. It is a bit unusual (crazy?) to have a conference between Christmas and New Year's Eve, but here was this NSF/NSA funded conference organized by Reinhard Laubenbacher who was at NMSU at the time. The whole event was scheduled around Bernd's ten one-hour talks which he delivered over five days. He wrote extensive lecture notes which included exercises. These lecture notes evolved into Bernd's blockbuster book {\it Gr\"obner bases and convex polytopes} \cite{S96} with the addition of four more chapters (chapters 3, 12, 13, and 14). They were an engaging summary of the state of the art in Gr\"obner fans and state polytopes of general polynomial ideals, Gr\"obner bases of toric ideals with applications in enumeration, sampling, and integer programming, universal and Graver bases of toric ideals, and regular triangulations. The last three talks were devoted to more recent developments: the second hypersimplex, $A$-graded algebras, and SAGBI bases. They were meant to stimulate new ideas and research. 

I got the lecture notes from Bernd probably late summer or early fall of 1994. He asked me to read the notes and give him my comments. In particular, he told me to ``do the exercises!" Here is the first exercise at the end of the first lecture which also appears as exercise (1) at the end of chapter 1 in \cite{S96}:

\vskip 0.2cm

Let $I$ be the ideal generated by the nine $2 \times 2$-minors of a $3 \times 3$-matrix of indeterminates. Find a universal Gr\"obner basis. How many distinct initial ideals $\mathrm{in}_\prec(I)$ are there?

\vskip 0.2cm

Even today I am surprised how anyone could have answered this question after reading only the notes of the first lecture (the result that a polynomial ideal can have only finitely many initial ideals, definition of universal Gr\"obner bases, a few examples, and the relation between weight vectors $\omega \in \R^n$ and initial ideals induced by them). If one had known the contents of the lectures on toric ideals and their universal Gr\"obner bases, one could have developed some ideas. Knowing the contents of the lecture on regular triangulations would have also helped. Of course, all of these needed to be done without the help of the (non-existing) software such as {\tt Gfan}. Luckily, I knew some of these things (had I really mastered them?) and I could get a start. 

The ideal $I$ is the toric ideal that is the defining ideal of $\PP^2 \times \PP^2$. The corresponding matrix $A$ is a $6 \times 9$ matrix
whose columns are $e_i \oplus e_j$, $1 \leq i,j \leq 3$, 
where $e_i$ form the standard unit basis of $\R^3$. This matrix is unimodular and therefore $\mathcal{U}_A = Gr_A$, and the latter can be computed with Theorem \ref{thm:graver}. Or you know that $\mathcal{U}_A$ is made up of the circuits of $A$ and they are in bijection with the cycles in the complete bipartite graph $K_{3,3}$. So there are $15$ squarefree binomials in $\mathcal{U}_A$, $9$ quadratics and $6$ cubics: $x^{u_1} - x^{v_1}, \, \ldots, \, x^{u_{15}} - x^{v_{15}}$. In the last step one needs to check whether there is a vector $\omega \in \R^9$ such that
$$ \omega \cdot u_1 > \omega \cdot v_1, \,\, \omega \cdot u_2 > \omega \cdot v_2, \, \cdots \, , \omega \cdot u_{15} > \omega \cdot v_{15}.$$
There are $2^{15}$ such linear inequality systems to check. Clearly, one needed to write some code using {\tt MAPLE}. When I went through all of these and got the answer ($108$; go ahead and try yourself) I was very happy. This was hard, but it was not completely unmanagable. One needs to be willing to extend and definitely do computations. I think this is how Bernd operates, in a nutshell. 

The Holiday Symposium was a huge success. There were some new faces in Las Cruces who have been contributing significantly to commutative algebra and computational algebraic geometry since then, such as Irena Peeva, Maurice Rojas, and Frank Sottile. Bernd's lectures went extremely well although he had caught some bug and therefore was sick. In characteristic fashion, Bernd announced exercise sessions and invited (!) \! all grad students and postdocs. One embarrassing anecdote for me happened during the first exercise session where we discussed the above exercise. At some point we landed at a place where we had to argue that circuits of a unimodular matrix have only $0, +1, -1$ entries. Bernd called me to the board. For some reason I froze. I could not reason for the life of me. He ``rescued" me by basically completing the answer. On the same evening he invited me to his hotel room and we chatted about various things. He did make a point to tell me that I should have been more prepared. But the comment was very neutral. It was not angry or shaming or condescending.  Also in a fine gesture, he wrote in the Introduction to {\it Gr\"obner bases and convex polytopes} that I ``did a particularly great job during the last round of proof-reading". The book is to this day the starting point for someone who is learning about toric ideals, Gr\"obner fans, state polytopes etc. As an interesting coincidence, Bernd taught a course in Berkeley on Gr\"obner bases and convex polytopes from the book in spring 2022 as he turned 60. Many students enrolled which shows that the book and the subject are still alive. 

\section{Degrees of Gr\"obner bases of toric ideals}

In late spring 1995 I took my A exam at Cornell. This is an oral PhD qualification exam where I had to also pitch to my committee (Bernd, Lou Billera, Dexter Kozen, Les Trotter) my plans for my thesis. By that time, Bernd and I had settled on the question of giving a better bound on the degrees of Gr\"obner or Graver basis elements of toric ideals. 

The starting point was Theorem \ref{thm:degree}. Motivated by the Eisenbud-Goto conjecture \cite{EG84} that was still a conjecture at the time and by computational observations, we thought that the bound in this theorem should be $(d+1)D(A)$. The proof of the theorem reveals that one could first try to show that the bound is given by the maximum degree  $\mathrm{cdeg}(A)$ of any circuit of $A$. I worked on this for a while, and together with Rekha we produced a counterexample. This counterexample is reported in \cite{S97} but it was already known by the time Bernd gave his lecture at the ``AMS Summer Institute in Algebraic Geometry at Santa Cruz" in July 1995 on which the publication was based. Here I  reproduce the counterexample as Bernd presented it.  
\begin{example} \label{ex:degree}
Let $B = \begin{bmatrix}  1 & 3 & 4 & 6 & 0 \\ 0 & 0 & 0 & -5 & 1 \end{bmatrix}$, and let $A = \Lambda(B)$. The ideal $I_A$ equals
$$
I_A  \, = \, \langle \,
 x_2 y_1^3 - x_1^3 y_2 , \,
x_3 y_1^4 - x_1^4 y_3 , \,
x_3^3 y_2^4 - x_2^4 y_3^3 ,\,
x_4 x_5^5 y_2^2 - x_2^2 y_4 y_5^5 ,\,
 x_4 x_5^5 y_1^6 - x_1^6 y_4 y_5^5,$$
$$
\underline{x_4^2 x_5^{10} y_3^3 - x_3^3 y_4^2 y_5^{10}  },\,\,
x_4 x_5^5 y_1^2 y_3 - x_1^2 x_3 y_4 y_5^5,\,
x_4 x_5^5 y_1^3 y_2 - x_1^3 x_2 y_4 y_5^5,\,
x_3 y_1 y_2 - x_1 x_2 y_3,
$$
$$
  x_1 x_4 x_5^5 y_2 y_3 \! - \! x_2 x_3 y_1 y_4 y_5^5,
  x_2^2 x_4 x_5^5 y_3^3 - x_3^3 y_2^2 y_4 y_5^5 ,
  x_1 x_3^2 y_2^3 - x_2^3 y_1 y_3^2 ,
  x_1^2 x_3 y_2^2 - x_2^2 y_1^2 y_3, $$
  $$
   x_2 x_4 x_5^5 y_1 y_3^2 - x_1 x_3^2 y_2 y_4 y_5^5 ,\,
  x_1^2 x_4 x_5^5 y_3^2 - x_3^2 y_1^2 y_4 y_5^5 ,\,\,
  x_4^2 x_5^{10} y_1 y_2 y_3^2 - x_1 x_2 x_3^2 y_4^2 y_5^{10} \,
\rangle. 
$$
The generators of $I_A$ form a Graver basis of $I_A$. The first six binomials including the underlined one correspond to the circuits of $A$. The underlined circuit attains $\mathrm{cdeg}(A) = 15$. However, the last binomial has degree $16 > 15$.
\end{example}
Although this example is nice it is not satisfactory since when we attempt to compute the underlined circuit using (\ref{eq:circuit}) we see that the constant $c=2$. In other words, the ``true" degree of this circuit is not $15$ but $30$. 
Therefore we define the true degree of a circuit to be the degree computed without dividing by the common factor $c$ in (\ref{eq:circuit}). With this we can refine the conjecture. 
\begin{conjecture} \cite[Conjecture 4.8]{S97} \label{conj:true-degree}
The degree of any Graver basis element of a toric ideal $I_A$ is at most the maximum true degree of any circuit of $A$. 
\end{conjecture}
A small progress due to Bernd was reported in my dissertation \cite{H97}. Recall that $H_\rho = \R^n_\rho \cap \ker_\Z(A)$. This rational polyhedral cone is generated by the circuits of $A$ conforming to the sign pattern $\rho$.
\begin{proposition}
If $H_\rho$ is a simplicial cone then the degree of any Graver basis element of $I_A$ contained in $H_\rho$ is bounded above by the 
maximum  true degree of any circuit in $H_\rho$. 
\end{proposition}
Progress on Conjecture \ref{conj:true-degree} stayed dormant for more than a decade. A result of Sonja Petrovi\'c showed that the conjecture holds for rational normal scrolls \cite{P08}. Then in 2011 Christos Tatakis and Apostolos Thoma provided a counterexample.
\begin{theorem} \cite{TT11}
There exists a toric ideal $I_{A(G)}$
associated to a graph $G$ for which 
Conjecture \ref{conj:true-degree} does not hold. 
\end{theorem}
\begin{proof}
Given an undirected graph $G$ on $d$ vertices and $n$ edges let $A(G)$ be the vertex-edge incidence matrix of $G$. The binomials in $I_{A(G)}$ corresponding to the circuits of $A(G)$ have been characterized in \cite{V95}. Similarly, Graver basis elements of $I_{A(G)}$ have been characterized in \cite{RTT12}. Now, let $G$ be a graph consisting of a cycle of length $s$ and $s$ odd cycles of length $\ell$, each attached to a vertex of the first cycle; see Figure 5 in \cite{TT11}. The even walk of length $\ell s + s$ traversing all edges of $G$ corresponds to a Graver basis element of degree $\ell s + s$. Longest cycles in this graph consist of two cycles of length $\ell$ joined by a path of length $s-1$. The corresponding circuits are the ones with the maximum true degree 
of $2(\ell + s -1)$. Now it is easy 
to see that $s(\ell + 1) > 2(\ell + s -1)$ since $s, \ell > 2$. 
\end{proof}
It turns out one can actually create really bad counterexamples. 
\begin{theorem} \cite[Theorem 4.5]{TT15} There exist a family of graphs $G_r$ where the maximum true degree of circuits in $I_{A(G_r)}$ grows linearly in $r$ whereas the maximum degree of Graver basis elements of $I_{A(G_r)}$ grows exponentially in $r$.   
\end{theorem}

\section{Initial ideals of toric ideals}
I wish to conclude this brief history of toric ideals in their first decade with one last result with which I was involved. The starting point was one more time integer programming problems $IP_{A,\omega}(b)$. In order to solve these problems, one usually passes to the {\it linear programming relaxation} where the constraint that the solution vectors are $x \in \N^n$ is replaced with the easier constraint $x \geq 0$. However, in the early days of integer programming Ralph~E.~Gomory introduced a different and more algebraic relaxation: let $\sigma$ be a subset of $\{1, \ldots, n\}$. Now replace the constraint that the solution vectors are $x \in \N^n$ with the constraints that $x_i \in \N$ for all $i \in \sigma$ and $x_i \in \Z$ for all $i \not \in \sigma$. This is known as a {\it group relaxation} of the integer programs $IP_{A,\omega}(b)$ \cite{G65}. 

Bernd had early done some observations relating group relaxations to the structure of the
monomial initial ideal $\mathrm{in}_\succ(I_A)$ of the toric ideal $I_A$ with respect to the term order $\succ$ induced by the vector $\omega$ \cite[Section 12.D]{S96}. His observations were pointing in the direction of studying minimal and embedded primes and primary decomposition of $\mathrm{in}_\succ(I_A)$. 

One more time, Rekha and I joined forces and dug into this. Two of our papers explored the connection of these initial ideals and their {\it standard pair decompositions} to group relaxations \cite{HT99a, HT03}. However, a third one is my favorite since it proved a surprising result \cite{HT99b}. 

By the time I got my PhD, there was only one general structural result known about $\mathrm{in}_\succ(I_A)$.
\begin{theorem} \cite[Corollary 8.9]{S96} The initial ideal $\mathrm{in}_\succ(I_A)$ is squarefree, i.e. it is a radical ideal, if and only if the corresponding regular triangulation of $A$ is unimodular.    
\end{theorem}
Was there anything interesting if $\mathrm{in}_\succ(I_A)$ was not radical? We were able to characterize the associated primes of $\mathrm{in}_\succ(I_A)$ in terms of polytopes we could construct from the matrix $A$, the vector $\omega \in \R^n$ that gave rise to the term order $\succ$, and certain right-hand-side vectors $b$ of the family
of integer programs $IP_{A,\omega}$. This characterization led to the following result that imposes strong structural restrictions on $\mathrm{in}_\succ(I_A)$. 
\begin{theorem} \cite[Theorem 3.1]{HT99a} If $P$ is an embedded prime of $\mathrm{in}_\succ(I_A)$, then $P$ contains an associated prime $Q$
of $\mathrm{in}_\succ(I_A)$ such that $\dim(Q) = \dim(P)+1$.    
\end{theorem}
In other words, the associated primes of the initial ideals of toric ideals come in chains: one can start with any embedded prime and reach a minimal prime by following a chain of associated primes and hitting every dimension along the way. Oddly, there has been no work that I am aware of studying this or related phenomena since last century. 

\section*{Acknowledgements}
In 2016, during a happy occasion, I told that I have three families: the nuclear family of my loving wife Ingrid and my son Orhan, the family of my colleagues at the SFSU Mathematics Department who embraced me for 25 years, and finally my first family from the Cornell days when I did not yet have the above two. Bernd, Rekha, and Jes\'us, I am tremendously thankful to you for accepting me into your warm and supportive circle. 

\bibliographystyle{amsplain}
\bibliography{sample}

\end{document}